%% file: master.tex
\documentclass[a4paper,fleqn,12pt,twoside,draft]{article}

\input{layout_commands}
\input{math_commands}
\input{draft_commands}
\input{my_commands}
\title{\titlefamily\Huge Small surfaces of Willmore type in Riemannian
manifolds}
\author{%
  \authname{Tobias Lamm
  \thanks{Partially supported by a PIMS Postdoctoral Fellowship.}}
  \authaddress
  {
    Department of Mathematics,
    University of British Columbia,
    1984 Mathematics Road,
    Vancouver, BC V6T 1Z2,
    Canada
  }
  \and
  \authname{Jan Metzger}
  \authaddress
  {
    Albert-Einstein-Institut,
    Am M\"uhlenberg 1,
    14476 Potsdam,
    Germany
  }
  \authaddress
  {
    Universit\"at Freiburg,
    Institut f\"ur Reine Mathematik,\\
    Eckerstr. 1,
    79104 Freiburg,
    Germany
  }
}
\date{}

\begin{document}
\hyphenation{}
\pagestyle{footnumber}
\maketitle
\thispagestyle{footnumber}
\begin{abst}%
  In this paper we investigate the properties of small surfaces of
  Willmore type in Riemannian manifolds. By \emph{small} surfaces we
  mean topological spheres contained in a geodesic ball of small
  enough radius. In particular, we show that if there exist such
  surfaces with positive mean curvature in the geodesic ball
  $B_r(p)$ for arbitrarily small radius $r$ around a point $p$ in the
  Riemannian manifold, then the scalar curvature must have a critical
  point at $p$.

  As a byproduct of our estimates we obtain a strengthened version of
  the non-existence result of Mondino \cite{Mondino:2008} that implies
  the non-existence of certain critical points of the Willmore
  functional in regions where the scalar curvature is non-zero.
\end{abst}
\input{intro}
%
\input{prelim}

%
\input{apriori}
%
\input{gradient}
%
\input{expansion}
\bibliographystyle{abbrv}
\bibliography{references}
\end{document}

%% file: math_commands.tex
\usepackage{amsmath}
\usepackage{amssymb}

\newcounter{mynumcounter}
	       {%
		 \begin{list}{(\roman{mynumcounter})\hspace*{\fill}}%
		   {
		     \setlength{\topsep}{0cm}
		     \setlength{\partopsep}{0cm}
		     \setlength{\itemsep}{0ex}
		     \setlength{\parsep}{0cm}
		     \setlength{\leftmargin}{0cm}
		     \setlength{\itemindent}{8mm}		   
		     \setlength{\labelsep}{3mm}
		     \setlength{\labelwidth}{5mm}
		     \usecounter{mynumcounter}
		   }%
	       }
	       {\end{list}}

\newcommand{\eps}{\varepsilon}
\newcommand{\del}{\partial}

\newcommand{\Vol}{\operatorname{Vol}}

\newcommand{\Ric}{\operatorname{Ric}}
\newcommand{\Scal}{\operatorname{Scal}}

\newcommand{\tr}{\operatorname{tr}}
\renewcommand{\div}{\operatorname{div}}

\newcommand{\IR}{\mathbf{R}}

\newcommand{\CU}{\mathcal{U}}
\newcommand{\CV}{\mathcal{V}}
\newcommand{\CW}{\mathcal{W}}

\newcommand{\al}{\alpha}

\newcommand{\ga}{\gamma}

\newcommand{\id}{\operatorname{id}}

\newcommand{\rmd}{\mathrm{d}}

\newcommand{\const}{\mathrm{const}}

\newcommand{\dmu}{\,\rmd\mu}

\newcommand{\ra}{\rangle}
\newcommand{\la}{\langle}
\newcommand{\inj}{\mathrm{inj}}

\newcommand{\Scalarcurv}[1]{\smash{\sideset{^{#1}}{}{\mathop\mathrm{Sc}\nolimits}}}

\newcommand{\ScalSig}{\!\Scalarcurv{\Sigma}}

\newcommand{\Acirc}{\hspace*{4pt}\raisebox{8.5pt}{\makebox[-4pt][l]{$\scriptstyle\circ$}}A}

\newcommand{\Acircbar}%
{\hspace*{4pt}\raisebox{8.5pt}{\makebox[-4pt][l]{$\scriptstyle\circ$}}\bar A}
\newcommand{\Kcircbar}%
{\hspace*{4pt}\raisebox{8.5pt}{\makebox[-4pt][l]{$\scriptstyle\circ$}}\bar K}

\newcommand{\half}{\tfrac{1}{2}}

%% file: my_commands.tex
\newcommand{\dvol}{\mathrm{d}V}
\renewcommand{\Scal}{\operatorname{Sc}}

%% file: intro.tex
\section{Introduction}
In a previous paper \cite{Lamm-Metzger-Schulze:2009} Willmore type surfaces were introduced and foliations of asymptotically flat
manifolds by such surfaces were studied. In this paper we turn to the local
situation and consider Willmore type surfaces in small geodesic balls
in Riemannian manifolds. The focus is on a priori estimates for such
surfaces under the assumption of positive mean curvature and a growth
condition for the Lagrange parameter. As an application of these
estimates we derive a necessary condition for the existence of such
surfaces.

By surfaces of Willmore type we mean surfaces $\Sigma$ that are
critical for the Willmore functional
\begin{equation*}
  \CW(\Sigma) = \frac 1 2 \int_\Sigma H^2 \dmu
\end{equation*}
subject to an area constraint $|\Sigma| = a$, where $a$ is some fixed
constant. These surfaces are solutions of the Euler-Lagrange equation 
\begin{equation}
  \label{eq:4}
  \Delta H + H |\Acirc|^2 + H \Ric(\nu,\nu) + \lambda H = 0
\end{equation}
where $\lambda\in \IR$ is the Lagrange parameter, $H$ the mean
curvature of $\Sigma$, $\Acirc$ denotes the traceless part of the
second fundamental form, $\Ric$ refers to the Ricci curvature of the
ambient manifold, $\nu$ is the normal of $\Sigma$ and $\Delta$ denotes
the Laplace-Beltrami operator on $\Sigma$. In particular these surfaces
are a generalization of Willmore surfaces that are critical for $\CW$
without constraint and therefore satisfy the equation
\begin{equation*}
  \Delta H + H |\Acirc|^2 + H \Ric(\nu,\nu) = 0.
\end{equation*}
We note here that there are other functionals that can be considered
as generalizations of the Willmore functional in Riemannian manifolds,
for example the functional $\CU$ introduced in
section~\ref{sec:preliminaries} could be used (see \cite{Weiner:1978}).

The precise statement of the main result of this paper is the following:
\begin{theorem}
  \label{thm:intro1}
  Assume that $(M,g)$ is a Riemannian manifold such that the curvature and
  the first two derivatives of the curvature are bounded. Then there exist
  $\eps_0>0$ and $C<\infty$, depending only on these bounds, with the
  following properties.

  Given $p$ in $M$ and assume that there is $r_0>0$ such
  that for each $r\in (0,r_0]$ there exists a surface $\Sigma_r$ of
  Willmore type in $B_r(p)$, that is on $\Sigma_r$ we have
  \begin{equation*}
    \Delta H + H|\Acirc|^2 + H\Ric(\nu,\nu) + H\lambda_r = 0,
  \end{equation*}
  such that in addition the following conditions are satisfied for $\eps<\eps_0$:
  \begin{enumerate}
  \item $\Sigma_r$ is a topological sphere, 
  \item $\lambda_r \geq -\eps / |\Sigma_r|$, and
  \item $H>0$ on $\Sigma_r$.
  \end{enumerate}
  Then 
  \begin{equation*}
    \lim_{r\to 0} |\lambda_r  + \tfrac{1}{3} \Scal(p)| = 0.
  \end{equation*}
  Here $\Scal(p)$ denotes the Scalar curvature of $M$ at the point
  $p$. Furthermore
  \begin{equation*}
    \nabla\Scal(p) = 0
  \end{equation*}
  where $\nabla\Scal(p)$ denotes the gradient of the scalar curvature
  of $M$ at $p$.
\end{theorem}
The first claim is proved in
section~\ref{sec:a-priori-estimates} as a consequence of the a priori
estimates for surfaces of Willmore type derived
there. Section~\ref{sec:gradient} is devoted to the proof of the
second claim.

For surfaces of constant mean curvature (CMC), that is surfaces
satisfying $H=\const$, analogous properties have been derived. Ye
showed that if there locally exists a regular foliation by CMC
surfaces near a point $p$, then $p$ is necessarily a critical point of
the scalar curvature.  For the detailed statement including the
technical conditions we refer to \cite[Theorem 2.1]{ye:1991} (see also
\cite{ye:1996a}). There are further results in this direction by Druet
\cite{druet:2002} and Nardulli \cite{nardulli:2009} where the
expansion of the isoperimetric profile of a Riemannian manifold is
computed. This computation shows that isoperimetric surfaces also
concentrate near critical points of the scalar curvature.

Indeed, it has been shown by Ye in \cite{ye:1991} that near
non-degenerate critical points of the scalar curvature there exist
spherical surfaces with arbitrarily large mean curvature, or
equivalently, arbitrarily small area. We expect that a similar
statement is true for surfaces of Willmore type, namely that near a
non-degenerate critical point of the scalar curvature there exist
surfaces of Willmore type with arbitrarily small area. We will address
this elsewhere.

An immediate corollary of theorem~\ref{thm:intro1}, is the following
strengthened version of the non-existence result of Mondino
\cite[Theorem 1.3]{Mondino:2008} for Willmore surfaces. These surfaces
are of Willmore type with multiplier $\lambda = 0$, and thus the
previous theorem is applicable.
\begin{corollary}
  Let $(M,g)$ be a Riemannian manifold as in theorem \ref{thm:intro1} and let $p\in M$. If
  $\Scal(p)\neq 0$ or $\nabla \Scal(p)\neq 0$ then there exists
  $r>0$ such that $B_{r}(p)$ does not contain spherical Willmore
  surfaces with positive mean curvature.
\end{corollary}
We conclude the paper with section~\ref{sec:expans-willm-funct}, where
the a priori estimates are used to calculate the expansion of the
Willmore functional on surfaces as in theorem~\ref{thm:intro1}. More
precisely, we show that for these surfaces we have
\begin{equation}
\label{expansion}
\CW(\Sigma)=8\pi-\frac{|\Sigma|}{3}\Scal(p)+O(r|\Sigma|).
\end{equation}
This is analogous to the expansion derived by Mondino
\cite[Proposition 3.1]{Mondino:2008} for perturbed spheres.
%%% Local Variables: 
%%% mode: latex
%%% TeX-master: "master"
%%% ispell-local-dictionary: "en_US"
%%% End: 

%% file: prelim.tex
\section{Preliminaries}
\label{sec:preliminaries}
In this section we describe our notation and we provide some tools in order to
analyze small surfaces in a Riemannian manifold.
\subsection{Notation}
We consider surfaces $\Sigma$ in a three dimensional Riemannian
manifold $(M,g)$, where $g$ denotes the metric on $M$. We denote by
$\nabla$ the induced Levi-Civita-Connection, by $\Ric$ its
Ricci-curvature and by $\Scal$ its scalar curvature.

If $p\in M$ and $\rho < \inj(M,g, p)$, the injectivity radius of $(M,g)$ at
$p$, we can introduce Riemannian normal coordinates on $B_\rho(p)$, the
geodesic ball of radius $\rho$ around $p$. These are given by the map
\begin{equation*}
  \Phi : B^E_\rho(0) \to B_\rho(p) : x \mapsto \exp_p(x)
\end{equation*}
where $B^E_\rho(0)$ is the Euclidean ball of radius $\rho$ in $\IR^3 \cong
T_p M$. In these coordinates the metric $g$ satisfies
\begin{equation}
  \label{eq:2}
  g = g^E + h
\end{equation}
where $g^E$ denotes the Euclidean metric and $h$ satisfies
\begin{equation}
  \label{eq:3}
  |x|^{-2}|h| + |x|^{-1}|\del h| + |\del^2 h| \leq h_0
\end{equation}
for all $x\in B^E_\rho(0)$. Here $h_0$ is a constant depending only on
the maximum of $|\Ric|$, $|\nabla\Ric|$ and $|\nabla^2\Ric|$ in
$B_\rho(p)$. More detailed expansions are not needed here but can be
found in \cite[Lemma V.3.4]{schoen-yau:1994}. For our purposes it is
sufficient to consider $M = B^E_\rho(0)$ to be equipped with the two
metrics $g$ and $g^E$. We will denote $B_\rho = B^E_\rho(0)$ in the
sequel.

If $\Sigma\subset B_\rho(p)$ is a surface, we denote its normal vector
by $\nu$, its induced metric by $\gamma$ and its second fundamental
form by $A$. The mean curvature of $\Sigma$ is denoted by $H = \tr_\ga
A$ and the traceless part of the second fundamental form by $\Acirc =
A - \half H \gamma$. Furthermore, $\dmu$ denotes the measure on
$\Sigma$. Note that also the Euclidean metric induces a full set of
geometric quantities on $\Sigma$, which will be distinguished by the
superscript $^E$, for example $\nu^E$, $A^E$, $H^E$, etc. All geometric quantities which we leave undecorated correspond to the
metric $g$.

%\subsection{Surface geometry}
%If $\Sigma\subset M$ is a hypersurface then its Scalar curvature
%$\ScalSig$ is given by the Gauss equation
%\begin{equation}
%  \label{eq:5}
%  \ScalSig = \Scal - 2\Ric(\nu,\nu) + \half H^2 - |\Acirc|^2.
%\end{equation}
%Recall that the Bochner identity \cite[Proposition
%4.15]{Gallot-Hulin-Lafontaine:1993} implies that for any smooth
%function $f\in C^\infty(\Sigma)$ we have
%\begin{equation*}
%  \int_\Sigma 2|(\nabla f)^\circ|^2 \dmu = \int_\Sigma (\Delta f)^2 -
%  2 \RicSig(\nabla f, \nabla f) \dmu,
%\end{equation*}
%where $\RicSig$ denotes the Ricci curvature of $\Sigma$.  Note that
%since $\Sigma$ is two dimensional, we have $\RicSig = \half \ScalSig
%\gamma$, and in view of the Gauss equation the Bochner identity thus
%becomes
%\begin{equation}
%  \label{eq:9}
%  \int_\Sigma 2|(\nabla f)^\circ|^2 + \half H^2 |\nabla f|^2 \dmu
%  =
%  \int_\Sigma (\Delta f)^2 + (|\Acirc|^2 -\Scal +
%  2\Ric(\nu,\nu))|\nabla f|^2 \dmu.
%\end{equation}

\subsection{The Willmore Functional}
Assume that $\Sigma\subset M$. Then we consider the \emph{Willmore
  functional} on $\Sigma$, that is the functional
\begin{equation*}
  \CW(\Sigma) = \frac 1 2 \int_\Sigma H^2 \dmu.
\end{equation*}
We say that a surface is of \emph{Willmore type with multiplier
  $\lambda\in\IR$} if it satisfies the equation
\begin{equation}
  \label{eq:33}
  \Delta H + H |\Acirc|^2 + H \Ric(\nu,\nu) + \lambda H = 0.
\end{equation}
Here $\Delta$ denotes the Laplace-Beltrami operator on $\Sigma$ and
$\Ric$ refers to the Ricci-curvature of the ambient metric $g$ as
before. Equation~\eqref{eq:33} arises as the Euler-Lagrange equation
for the the following variational problem:
\begin{equation*}
  \begin{cases}
    \text{Minimize} & \CW(\Sigma) \\
    \text{subject to} & |\Sigma| = a
  \end{cases}
\end{equation*}
where $a$ is a given constant. The parameter $\lambda$ in~\eqref{eq:33}
is then just the Lagrange-parameter of the critical point. For a
derivation of this and further motivation we refer to
\cite{Lamm-Metzger-Schulze:2009}.

Denoting by $\ScalSig$ the scalar curvature of $\Sigma$, the Gauss
equation implies that
\begin{equation}
  \label{eq:5}
  \ScalSig = \Scal - 2\Ric(\nu,\nu) + \half H^2 - |\Acirc|^2.
\end{equation}
Integrating this equation on $\Sigma$ yields the identity
\begin{equation}
  \label{eq:17}
  \CW(\Sigma) = 8\pi(1-q(\Sigma)) + \CU(\Sigma) + \CV(\Sigma)
\end{equation}
where $q(\Sigma)$ denotes the genus of $\Sigma$,
\begin{align*}
  &\CU(\Sigma) = \int_\Sigma |\Acirc|^2\dmu, \qquad\text{and} \\
  &\CV(\Sigma) = 2\int_\Sigma G(\nu,\nu) \dmu.
\end{align*}
Here $G = \Ric - \half \Scal g$ denotes the Einstein tensor of $M$.
This splitting was used in~\cite{Lamm-Metzger-Schulze:2009} to obtain
a priori estimates for the position of Willmore type surfaces in
asymptotically flat manifolds and shall also play an important role in
section~\ref{sec:gradient}.

\subsection{Small surfaces}
Since we compare the geometry of a surface $\Sigma$ with respect to
the ambient metrics $g$ and $g^E$ we need the following lemma.
\begin{lemma}
  \label{thm:metric_change_estimate}
  Let $g=g^E+h$ on $B_\rho$ be given. Then there exists a constant $C$
  depending only on $\rho$ and $h_0$ from equation~\eqref{eq:3} such
  that for all surfaces $\Sigma\subset B_r$ with $r<\rho$ we have
  \begin{align*}
    &|\gamma -\gamma^E| \leq C |x|^2, \\
    &|\dmu - \dmu^E| \leq C|x|^2, \\
    &|\nu - \nu^E | \leq C |x|^2, \quad\text{and} \\
    &|A- A^E| \leq C (|x| + |x|^2|A|).
  \end{align*}
\end{lemma}
In the sequel we will use the big-$O$ notation. By the statement $f=
O(r^\al)$ we mean that for any $r_0>0$ there exists a constant
$C<\infty$ such that $|f| \leq Cr^\al$ provided that $r<r_0$.

Observe that the area of a surface in $B_\rho$ is bounded in terms of
$\rho$ and the Willmore functional. This lemma is a slight
generalization of \cite[Lemma 1.1]{simon:1993}.
\begin{lemma}
  \label{thm:area_estimate}
  Let $g=g^E+h$ on $B_\rho$ be given. Then there exists $0<\rho_0
  <\rho$ and a constant $C$ depending only on $\rho$ and $h_0$ such
  that for all surfaces $\Sigma\subset B_r$ with $r<\rho_0$ we have
  \begin{equation*}
    |\Sigma| \leq C r^2 \int_\Sigma H^2\dmu.
  \end{equation*}
\end{lemma}
\begin{proof}
  Let $\Sigma \subset B_r$ be a hypersurface for some $r\leq \rho$.
  We consider the position vector field $x$ on $B_\rho$. Then we have
  \begin{equation*}
    \div_\Sigma x = 2 + O(|x|)
  \end{equation*}
  where $\div_\Sigma$ means the tangential divergence along
  $\Sigma$. Integrating this relation yields
  \begin{equation*}
    2 |\Sigma| = \int_\Sigma \div_\Sigma x \dmu + |\Sigma| O(r).
  \end{equation*}
  Since
  \begin{equation*}
    \int_\Sigma \div_\Sigma x \dmu
    =
    \int_\Sigma H \la x, \nu \ra \dmu
    \leq
    \left(\int_\Sigma H^2 \dmu\right)^{1/2}
    \left(\int_\Sigma |\la x, \nu \ra|^2\dmu\right)^{1/2}
  \end{equation*}
  and $|\la x,\nu \ra| \leq r$, we find that
  \begin{equation*}
    |\Sigma|
    \leq
    r |\Sigma|^{1/2} \left(\int_\Sigma H^2 \dmu\right)^{1/2}
    + Cr |\Sigma|.
  \end{equation*}
  Now we can fix $\rho_0$ small, so that for all $0< r<\rho_0$ the
  the second term on the right can be absorbed to the left. This
  yields the claimed inequality.
\end{proof}
For the subsequent curvature estimates we also need a version of the
Michael-Simon-Sobolev inequality. It follows from the Euclidean
version of the inequality \cite{Michael-Simon:1973} in conjunction
with lemma~\ref{thm:metric_change_estimate} to change to the
respective quantities to the $g$-metric.
\begin{lemma}
  \label{thm:michael-simon-soboloev}
  Let $g=g^E+h$ on $B_\rho$ be given. Then there exists $0<\rho_0
  <\rho$ and a constant $C$ depending only on $\rho$ and
  $h_0$ such that for all surfaces
  $\Sigma\subset B_{\rho_0}$ with $\|H\|_{L^2(\Sigma)}<\infty$ and all $f\in C^\infty(\Sigma)$ we have
  \begin{equation*}
    \left(\int_\Sigma f^2 \dmu \right)^{1/2}
    \leq
    C \int_\Sigma |\nabla f| + |Hf| \dmu.
  \end{equation*}
\end{lemma}
\subsection{Almost umbilical surfaces}
Subsequently it is necessary to approximate a given surface $\Sigma$
by a Euclidean sphere. The main tool will be the following theorem
from \cite{DeLellis-Muller:2005} and \cite{DeLellis-Muller:2006}. We
denote the $L^2$-norm of the trace free part of the second fundamental
form by
\begin{equation*}
  \|\Acirc^E\|^2_{L^2(\Sigma,\gamma^E)}
  =
  \int_\Sigma |\Acirc^E|_E^2\dmu^E,  
\end{equation*}
where all geometric quantities are with respect to the Euclidean
background. In addition we denote by
\begin{equation*}
  \|\Acirc\|^2_{L^2(\Sigma,\gamma)}
  =
  \int_\Sigma |\Acirc|^2\dmu,  
\end{equation*}
the norm of the same tensor, where all geometric quantities are
calculated with respect to the background metric $g$. The following
theorem is a purely Euclidean theorem.
\begin{theorem}
  \label{thm:delellis-muller}
  There exists a universal constant $C$ with the following properties.
  Assume that $\Sigma\subset\IR^3$ is a surface with
  $\|\Acirc^E\|^2_{L^2(\Sigma,\gamma^E)}< 8\pi$.  Let $R^E :=
  \sqrt{|\Sigma|^E/4\pi}$ be the Euclidean area radius of $\Sigma$ and
  $a^E := |\Sigma|_E^{-1} \int_\Sigma x \dmu^E$ be the Euclidean
  center of gravity. Then there exists a conformal map $\psi:
  S:=S_{R^E}(a^E) \to \Sigma\subset\IR^3$ with the following
  properties. Let $\gamma^S$ be the standard metric on $S$, $N$ the
  Euclidean normal vector field and $h$ the conformal factor, that is
  $\psi^*\gamma^E = h^2 \gamma^S$. Then the following estimates hold
  \begin{align*}
    \| H^E - 2/R^E \|_{L^2(\Sigma,\gamma^E)}
    &\leq
    C \|\Acirc^E\|_{L^2(\Sigma,\gamma^E)}
    \\
    \|\psi - \id_S \|_{L^\infty(S)}
    &\leq
    C R^E \|\Acirc^E\|_{L^2(\Sigma,\gamma^E)}
    \\
    \|h^2 -1\|_{L^\infty(S)}
    &\leq
    C R^E \|\Acirc^E\|_{L^2(\Sigma,\gamma^E)}
    \\
    \|N - \nu^E\circ\psi\|_{L^2(\Sigma,\gamma^E)}
    &\leq
    C R^E \|\Acirc^E\|_{L^2(\Sigma,\gamma^E)}.
  \end{align*}
\end{theorem}
To apply the previous theorem we need to estimate
$\|\Acirc^E\|_{L^2(\Sigma,\gamma^E)}$ in terms of
$\|\Acirc\|_{L^2(\Sigma,\gamma)}$. This is the content of the
following lemma.
\begin{lemma}
  \label{thm:euclidean-umbilic}
  Let $g=g^E+h$ on $B_\rho$ be given. Then there exists $0<\rho_0 <\rho$
  and a constant $C$ depending only on $\rho$ and $h_0$ such that for all surfaces $\Sigma\subset
  B_r$ with $r<\rho_0$ we have
  \begin{equation*}
    \begin{split}
      \|\Acirc^E\|^2_{L^2(\Sigma,\gamma^E)}
      % &\leq
      % C \|\Acirc\|^2_{L^2(\Sigma,\gamma)}
      % + Cr^2\big( 1 + \|H\|_{L^2(\Sigma,\gamma)}\big)
      % \|\Acirc\|_{L^2(\Sigma,\gamma)}    
      % \\
      &\leq
      C \|\Acirc\|^2_{L^2(\Sigma,\gamma)} + Cr^4 \|H\|^2_{L^2(\Sigma,\gamma)}
    \end{split}
  \end{equation*}
\end{lemma}
\begin{proof}
  This is a straightforward consequence of
  lemma~\ref{thm:metric_change_estimate} and the Cauchy-Schwarz
  inequality.
\end{proof}

%%% Local Variables: 
%%% mode: latex
%%% TeX-master: "master"
%%% ispell-local-dictionary: "en_US"
%%% End: 

%% file: apriori.tex
\section{A priori estimates}
\label{sec:a-priori-estimates}
A crucial ingredient in the proof of theorem \ref{thm:intro1} is an estimate for the $L^2$-norm
of the traceless part of the second fundamental form of $\Sigma$. This
allows us to control the shape of the surface $\Sigma$ in view of
theorem~\ref{thm:delellis-muller}.

Throughout this section we assume that the metric $g = g^E +h$ is
fixed on $B_\rho$. We furthermore assume that $\rho$ is chosen so
small that lemmas~\ref{thm:metric_change_estimate},
\ref{thm:area_estimate} and \ref{thm:michael-simon-soboloev} can be
applied to any surface in $B_\rho$. We allow $\rho$ to
shrink as it becomes necessary. All surfaces we consider here are of
Willmore type, i.e.\ they satisfy equation~\eqref{eq:4} for some
$\lambda$ and are furthermore contained in $B_\rho$.

Subsequently all constants $C$ may depend on $\rho$ and $h_0$ without further notice. In addition these
constants are allowed to change from line to line.

\subsection{The initial estimate for $\CU$}
\label{sec:initial-estimate-U}
\begin{lemma}
  \label{thm:initial-estimate}
  Let $g=g^E+h$ on $B_\rho$ be given. Then for each $\eps_0\in[0,1)$
  there exists $0<\rho_0 <\rho$ and a constant $C$ with the
  following properties.  If $\eps<\eps_0$ and $\Sigma$ is of Willmore type with
  multiplier $\lambda$ in $B_{r}$ with $r<\rho_0$ and
  \begin{enumerate}
  \item $\Sigma$ is a topological sphere,
  \item $\lambda \geq -\eps/|\Sigma|$,
  \item $H>0$ on $\Sigma$.
  \end{enumerate}
  Then
  \begin{align*}
    &\int_\Sigma |\Acirc|^2 + |\nabla \log H|^2 \dmu
    \leq C r^2 + \eps, \\
    &\left|\int_\Sigma H^2 \dmu -16\pi\right| \leq C(\eps+r^2),\\
    &|\Sigma| \leq C r^2,\quad\text{and}\\
    &|\lambda|\leq C(1+\eps/|\Sigma|).
  \end{align*}
\end{lemma}
Note that in equation~\eqref{eq:4} the term $\Delta H$ scales like
$|\Sigma|^{-3/2}$ so that the assumption on $\lambda$ implies that the
term $\lambda H$ is of the same order of magnitude as this
leading order term but small in comparison.

The proof is similar to the proof of lemmas 3.1 and 3.3 in
\cite{Lamm-Metzger-Schulze:2009} although the role of the individual
terms is somewhat different.
\begin{proof}
  Multiplying equation~\eqref{eq:4} by $H^{-1}$ and
  integrating by parts gives
  \begin{equation}
    \label{eq:6}
    \int_\Sigma |\nabla \log H|^2 + |\Acirc|^2 + \Ric(\nu,\nu) +
    \lambda \dmu = 0.   
  \end{equation}
  In view of the assumption on $\lambda$ and the fact that $\Ric$ is
  bounded this yields the estimate
  \begin{equation}
    \label{eq:7}
    \int_\Sigma |\nabla \log H|^2 + |\Acirc|^2 \dmu
    \leq C|\Sigma| + \eps.    
  \end{equation}
  Inserting this back into \eqref{eq:6} we get
  \begin{equation*}
 |\lambda|\leq C(1+\eps |\Sigma|^{-1}),
  \end{equation*}
  which yields the last claim of the lemma.

  Integrating \eqref{eq:5} over $\Sigma$ and using \eqref{eq:7} and
  the Gauss-Bonnet theorem we find that
  \begin{equation*}
    \left|\frac 1 2 \int_\Sigma H^2 \dmu-8\pi\right|
      \leq C(|\Sigma| + \eps).
   \end{equation*}
  In view of the area estimate from lemma~\ref{thm:area_estimate},
  this yields an estimate of the form
  \begin{equation*}
    |\Sigma|
    \leq
    C r^2 \int_\Sigma H^2 \dmu
    \leq
    C r^2(1 + |\Sigma|).
  \end{equation*}
  If $r<\rho_0$ is chosen small enough, we can absorb the term
  $Cr^2|\Sigma|$ on the right to the left and obtain the estimates
  \begin{equation*}
    \left|\int_\Sigma H^2 \dmu-16\pi\right| \leq C(\eps+r^2)
  \end{equation*}
  and 
  \begin{equation*}
    |\Sigma|
    \leq
    Cr^2
  \end{equation*}
  which are the second and third claims. Plugging this into
  estimate~\eqref{eq:7}, we obtain the remaining estimate.
\end{proof}
\subsection{An improved estimate for $\CU$}
\label{sec:improved-estimate-U}
The initial estimate from lemma \ref{thm:initial-estimate} allows to
apply the a priori estimates from section 3
in~\cite{Lamm-Metzger-Schulze:2009} to get higher order estimates and
to improve on the initial estimate.
\begin{theorem}
  \label{thm:integral-estimate1}
  Let $g=g^E+h$ on $B_\rho$ be given. Then there exist $\eps>0$,
  $0<\rho_0 <\rho$ and a constant $C$ with the following
  properties.  If $\eps<\eps_0$ and $\Sigma$ is of Willmore type with multiplier
  $\lambda$ in $B_{r}$ with $r<\rho_0$ and
  \begin{enumerate}
  \item $\Sigma$ is a topological sphere,
  \item $\lambda \geq -\eps/|\Sigma|$,
  \item $H>0$ on $\Sigma$.
  \end{enumerate}
  Then
  \begin{equation*}
    \int_\Sigma
    \frac{|\nabla^2 H|^2}{H^2}
    + |\nabla A|^2 +
    |A|^2|\Acirc|^2 \dmu
    \leq
    C \int_\Sigma |\omega|^2 + \big(\Ric(\nu,\nu) + \lambda\big)^2 \dmu.
  \end{equation*}
  Here $\omega = \Ric(\nu,\cdot)^T$ is a the tangential/normal
  component of the ambient Ricci curvature.
\end{theorem}
\begin{proof}
  This is a consequence of the calculation in section 3 of
  \cite{Lamm-Metzger-Schulze:2009}. Note that the calculation there
  makes use of the fact that $\|\Acirc\|_{L^2} + \|\nabla\log
  H\|_{L^2}$ can be made arbitrarily small (cf.\ lemma 3.8 there) and
  this is where the initial estimate from section
  \ref{sec:initial-estimate-U} enters here. In particular, theorem 3.9 in \cite{Lamm-Metzger-Schulze:2009}
  implies that in the local situation we have
  \begin{equation}
    \label{eq:7c}
    \begin{split}
      &\int_\Sigma
      \frac{|\nabla^2 H|^2}{H^2}
      + |\nabla A|^2 +
      |A|^2|\Acirc|^2 \dmu
      \\
      &\quad
      \leq
      C \int_\Sigma |\omega|^2 + \big(\Ric(\nu,\nu) + \lambda\big)^2
      \dmu
      \\
      &\quad\quad
      + C \sup_{B_\rho} |\Ric| \ \int_\Sigma |\Acirc|^2 + |\nabla\log
      H|^2 \dmu.
    \end{split}
  \end{equation}
  Next we use the Michael-Simon-Sobolev inequality to estimate
  \begin{equation}
    \label{eq:7a}
    \int_\Sigma|\Acirc|^2 \dmu
    \leq
    C\left(\int_\Sigma |\nabla\Acirc| + H |\Acirc| \dmu\right)^2
    \leq
    C |\Sigma| \int_\Sigma |\nabla A|^2 + H^2 |\Acirc|^2 \dmu.
  \end{equation}
  Similarly we get
  \begin{equation*}
    \begin{split}
      \int_\Sigma |\nabla\log H|^2 \dmu
      &\leq
      \left(\int_\Sigma \frac{|\nabla^2 H|}{H}+|\nabla \log
        H|^2+|\nabla H|\dmu\right)^2
      \\
      &\leq
      C|\Sigma| \int_\Sigma \frac{|\nabla H|^2}{H^2}+|\nabla A|^2+|\nabla\log H|^4\dmu.
    \end{split}
  \end{equation*}
  Applying the Michael-Simon Sobolev inequality once more we have
  \begin{equation*}
    \begin{split}
      &\int_\Sigma|\nabla\log H|^4\dmu      
      \\
      &\quad\leq
      C\left(\int_\Sigma \frac{|\nabla^2 H|}{H}|\nabla \log H|+|\nabla
        \log H|^3+H|\nabla \log H|^2\dmu\right)^2 
      \\
      &\quad\leq
      C\|\nabla \log H\|_{L^2(\Sigma)}^2\int_\Sigma \frac{|\nabla^2 H|^2}{H^2}+|\nabla \log H|^4+|\nabla A|^2\dmu.
    \end{split}
  \end{equation*}
  Using lemma \ref{thm:initial-estimate} we know that for $\eps$ and $r<r_0$ small enough we get
  \begin{equation}
    \label{eq:7d}
    \int_\Sigma|\nabla\log H|^4\dmu \le C ||\nabla \log H||_{L^2(\Sigma)}^2\int_\Sigma \frac{|\nabla^2 H|^2}{H^2}+|\nabla A|^2\dmu.
  \end{equation}
  Inserting this into the above estimate for $\int_\Sigma |\nabla\log H|^2 \dmu$ we conclude
  \begin{equation}
    \label{eq:7b}
    \int_\Sigma |\nabla\log H|^2 \dmu\leq C|\Sigma|  \int_\Sigma  \frac{|\nabla^2 H|^2}{H^2}+|\nabla A|^2\dmu.
  \end{equation}
  Hence we see that for $r<\rho_0$ small enough, we can absorb the second term on the right hand side of \eqref{eq:7c}.
\end{proof}
The remaining task is to estimate the term on the right hand side in
theorem~\ref{thm:integral-estimate1}. We start with the following
calculation.
\begin{lemma}
  \label{thm:calc-ric}
  Assume that the metric $g=g^E +h$ on $B_\rho$ is given. Then there
  exists a constant $C$ such that for all
  surfaces $\Sigma \subset B_r$ we have 
  \begin{equation*}
    \left|\int_\Sigma \Ric(\nu,\nu)\dmu - \frac{|\Sigma|}{3} \Scal(0)
    \right|
    \leq
    C|\Sigma|  \big( \|\Acirc\|_{L^2(\Sigma)} + r\big)  
  \end{equation*}  
\end{lemma}
\begin{proof}
  Note that if either $\|\Acirc\|_{L^2(\Sigma)}$ or $r$ is large, the estimate
  is trivially satisfied, so that it is sufficient to show it in the
  case where theorem~\ref{thm:delellis-muller} is applicable and we
  are furthermore allowed to assume that $0<r<1$. We use 
  theorem~\ref{thm:delellis-muller} to approximate $\Sigma$ by a
  Euclidean sphere $S = S_{a^E}(R^E)$ with $a^E\in B_r(0)$ and $R^E =
  \sqrt {|\Sigma|^E/4\pi}$. Since
  \begin{equation*}
    \big||\Sigma|^E - |\Sigma|\big|
    \leq
    \int_\Sigma |\dmu^E -\dmu|\dmu
    \leq
    Cr^2|\Sigma|
  \end{equation*}
  we infer that
  \begin{equation*}
    |R^E - R| \leq CrR
  \end{equation*}
  where $R = \sqrt{|\Sigma|/4\pi}$. It is well known that
  \begin{equation*}
    \int_S \Ric_0(N,N) \dmu^E
    =
    \frac{|\Sigma|^E}{3} \Scal(0)
  \end{equation*}
  where we use the notation of theorem~\ref{thm:delellis-muller},
  that is $N$ is the Euclidean normal of $S$. Furthermore $\Ric_0$
  denotes the Ricci-tensor of $M$ evaluated at the origin.

  The first step is to estimate
  \begin{equation*}
    \begin{split}
      &\left|\int_S \Ric(N,N) \dmu^E - \int_S \Ric_0(N,N)
        \dmu^E\right|
      \\
      &\quad
      \leq
      |\Sigma|^E \sup_{p\in S} |\Ric_p(N,N) - \Ric_0(N,N)|     
      \leq C|\Sigma|^E r
    \end{split}
  \end{equation*}
  and therefore
  \begin{equation*}
    \left|\int_S \Ric(N,N) \dmu^E - \frac{|\Sigma|^E}{3}
      \Scal(0)\right|
    \leq C r |\Sigma|.
  \end{equation*}
  In the next step we estimate $\int_\Sigma \Ric(\nu,\nu)$ in terms of
  $\int_S\Ric(N,N)\dmu^E$. To this end note that
  \begin{equation*}
    \left| \int_\Sigma \Ric(\nu,\nu)\dmu - \int_\Sigma \Ric(\nu,\nu)
      \dmu^E \right|
    \leq C |\Sigma| r^2.
  \end{equation*}
  The resulting integral can be evaluated using the conformal
  parametrization $\psi: S\to\Sigma$ from
  theorem~\ref{thm:delellis-muller}. We can express
  \begin{equation*}
    \int_\Sigma \Ric(\nu,\nu) \dmu^E = \int_S
    \Ric\circ\psi(\nu\circ\psi, \nu\circ\psi) h^2\dmu^E.
  \end{equation*}
  The estimates of theorem~\ref{thm:delellis-muller} and the
  Cauchy-Schwarz inequality imply that
  \begin{equation*}
    \begin{split}
      &\left|\int_\Sigma \Ric(\nu,\nu) \dmu^E -
        \int_S\Ric(N,N)\dmu^E\right|
      \\
      &\quad
      \leq
      \left| \int_S \big(\Ric\circ\psi -\Ric)(\nu\circ\psi, \nu\circ\psi)
        h^2 \dmu^E \right|
      \\
      &\qquad
      +
      \left| \int_S \Ric(\nu\circ\psi -N, \nu \circ \psi) h^2 \dmu^E
      \right|
      \\
      &\qquad
      +
      \left| \int_S \Ric(N, \nu \circ \psi-N) h^2 \dmu^E
      \right|
      +
      \left| \int_S \Ric(N,N) (h^2 -1)\dmu^E
      \right|
      \\
      &\quad
      \leq
      C|\Sigma| \|\psi-\id\|_{L^\infty(S)}
      + C|\Sigma|^{1/2} \|\nu\circ\psi - N\|_{L^2(S)}
      + C|\Sigma| \|h^2-  1\|_{L^\infty(S)}
      \\
      &\quad
      \leq
      C|\Sigma| \|\Acirc^E\|_{L^2(\Sigma,\gamma^E)}.
    \end{split}
  \end{equation*}
  In combination with lemma~\ref{thm:euclidean-umbilic} we infer
  \begin{equation*}
    \left|\int_\Sigma \Ric(\nu,\nu) \dmu^E -
      \int_S\Ric(N,N)\dmu^E\right|
    \leq
    C|\Sigma| \big( \|\Acirc\|_{L^2(\Sigma)} + r^2\big).
  \end{equation*}
  Collecting all the above estimates results in the estimate
  \begin{equation*}
    \left|\int_\Sigma \Ric(\nu,\nu)\dmu - \frac{|\Sigma|}{3} \Scal(0)
    \right|
    \leq
    C|\Sigma|  \big( \|\Acirc\|_{L^2(\Sigma)} + r\big)     
  \end{equation*}
  which is precisely the claim.
\end{proof}
In the following lemma we derive an estimate for the Lagrange parameter
$\lambda$.
\begin{lemma}
  \label{thm:lambda_estimate}
  Assume that the metric $g=g^E +h$ on $B_\rho$ is given. Then there
  exist $\eps_0$, $r_0<\rho$ and a constant $C$ such that all surfaces $\Sigma \subset
  B_r$ as in the statement of
  theorem~\ref{thm:integral-estimate1} with $\eps<\eps_0$ and $r<r_0$ satisfy 
  \begin{equation*}
    \big| \lambda + \tfrac{1}{3} \Scal(0) \big|
    \leq
    C|\Sigma|^{-1}\big(\|\Acirc\|_{L^2(\Sigma)}^2 + \|\nabla \log H\|_{L^2(\Sigma)}^2
    \big)
    + C r.
  \end{equation*}
  In particular
  \begin{equation*}
    |\lambda|
    \leq
    C|\Sigma|^{-1}\big(\|\Acirc\|_{L^2(\Sigma)}^2 + \|\nabla \log H\|_{L^2(\Sigma)}^2
    \big) + C.
  \end{equation*}
\end{lemma}
\begin{proof}
  Equation~\eqref{eq:6} implies that
  \begin{equation}
    \label{eq:8}
    \left| \lambda + \frac{1}{|\Sigma|} \int_\Sigma
      \Ric(\nu,\nu)\dmu\right|
    \leq
    |\Sigma|^{-1} \big(\|\Acirc\|_{L^2(\Sigma)}^2 + \|\nabla \log H\|_{L^2(\Sigma)}^2
    \big).
  \end{equation}
  Apply lemma~\ref{thm:calc-ric} to calculate 
  \begin{equation*}
    \left|\frac{1}{|\Sigma|}\int_\Sigma \Ric(\nu,\nu)\dmu - \frac{1}{3} \Scal(0)
    \right|
    \leq
    C \big( \|\Acirc\|_{L^2(\Sigma)} + r\big)   
  \end{equation*}
  and note that
  \begin{equation*}
    \|\Acirc\|_{L^2(\Sigma)}
    =
    |\Sigma|^{1/2} |\Sigma|^{-1/2}\|\Acirc\|_{L^2(\Sigma)}
    \leq
    \half |\Sigma| + \half |\Sigma|^{-1} \|\Acirc\|_{L^2(\Sigma)}^2.
  \end{equation*}
  Since $|\Sigma| \leq Cr^2$ by
  lemma~\ref{thm:initial-estimate}, this implies the claim in
  combination with equation~\eqref{eq:8}.
\end{proof}
\begin{theorem}
  \label{thm:main-estimate}
  Assume that the metric $g=g^E +h$ on $B_\rho$ is given. Then there
  exist $\eps_0>0$, $r_0<\rho$ and a constant $C$ such that all
  surfaces $\Sigma \subset B_r$ as in the statement of
  theorem~\ref{thm:integral-estimate1} with $\eps<\eps_0$ and $r<r_0$ satisfy
  \begin{equation*}
    \int_\Sigma \frac {|\nabla^2H|^2}{H^2} + |\nabla A|^2 +
    |A|^2|\Acirc|^2 \dmu
    \leq
    C|\Sigma|.
  \end{equation*}  
\end{theorem}
\begin{proof}
  In view of theorem~\ref{thm:integral-estimate1} and the fact that
  $\Ric$ and $\omega$ are bounded, we infer the estimate
  \begin{equation}
    \label{eq:10}
    \int_\Sigma \frac {|\nabla^2H|^2}{H^2} + |\nabla A|^2 +
    |A|^2|\Acirc|^2 \dmu
    \leq
    C|\Sigma|(1 + \lambda ^2).     
  \end{equation}
  The crucial term to estimate thus is $\lambda^2|\Sigma|$. We use the
  estimate from lemma~\ref{thm:lambda_estimate} to  get
  \begin{equation}
    \label{eq:11}
    \lambda^2|\Sigma|
    \leq
    C|\Sigma|^{-1} \big( \|\Acirc\|_{L^2(\Sigma)}^4 + \|\nabla \log
    H\|_{L^2(\Sigma)}^4) + C|\Sigma|.
  \end{equation}
  Combining \eqref{eq:7a} and \eqref{eq:7b} with \eqref{eq:10}
  and \eqref{eq:11}, we infer
  \begin{equation}
    \label{eq:12}
    \begin{split}
      &\int_\Sigma \frac {|\nabla^2H|^2}{H^2} + |\nabla A|^2 +
      |A|^2|\Acirc|^2 \dmu
      \\
      &\quad
      \leq
      C|\Sigma|
      +
      C|\Sigma| \left( \int_\Sigma \frac {|\nabla^2H|^2}{H^2} + |\nabla A|^2 +
        |A|^2|\Acirc|^2 \dmu\right)^2.
    \end{split}
  \end{equation}
  To proceed note that by equation~\eqref{eq:10} and lemma \ref{thm:initial-estimate} we find that
  \begin{equation*}
    \int_\Sigma \frac {|\nabla^2H|^2}{H^2} + |\nabla A|^2 +
    |A|^2|\Acirc|^2 \dmu
    \leq
    C|\Sigma|(1+\lambda^2)
    \leq C(|\Sigma|+\eps^2|\Sigma|^{-1})
  \end{equation*}
  Using this in equation~\eqref{eq:12} to estimate part of the right
  hand side, we get
  \begin{equation*}
    \begin{split}
      &\int_\Sigma \frac {|\nabla^2H|^2}{H^2} + |\nabla A|^2 +
      |A|^2|\Acirc|^2 \dmu
      \\
      &\quad
      \leq
      C|\Sigma|
      +
      C (|\Sigma|^2+\eps^2) \left( \int_\Sigma \frac {|\nabla^2H|^2}{H^2} + |\nabla A|^2 +
        |A|^2|\Acirc|^2 \dmu\right).
    \end{split}    
  \end{equation*}
  Thus choosing $0<\eps$ and $r<r_0$ small enough, we can absorb the second term on
  the right to the left and infer the claimed estimate.
\end{proof}
\begin{corollary}
  \label{thm:lambda-bounded}
  Assume that the metric $g=g^E +h$ on $B_\rho$ is given. Then there
  exist $\eps_0>0$, $0<r_0<\rho$ and a constant $C$ and such that all
  surfaces $\Sigma \subset B_r$ as in the statement of
  theorem~\ref{thm:integral-estimate1} with $\eps<\eps_0$ and $r<r_0$ satisfy
  \begin{equation*}
    \|\Acirc\|_{L^2(\Sigma)} + \|\nabla\log H\|_{L^2(\Sigma)} \leq C|\Sigma|,
  \end{equation*}
  and
  \begin{equation*}
    |\lambda + \tfrac{1}{3}\Scal(0)| \leq Cr.
  \end{equation*}
\end{corollary}
\begin{proof}
  The first claim follows from \eqref{eq:7a}, \eqref{eq:7b} and
  theorem \ref{thm:main-estimate}, whereas the second claim follows
  from the first one in view of lemma~\ref{thm:lambda_estimate}.
\end{proof}
Note that this corollary yields the first claim in
theorem~\ref{thm:intro1}.

\subsection{Estimates in the $L^\infty$-norm}
To proceed further we need an estimate for the size of $H^{-1}$ in the
$L^\infty$-norm. To this end we recall lemma 4.7
from~\cite{Lamm-Metzger-Schulze:2009}.
\begin{lemma}
  \label{thm:sup-norm-from-w2-2}
  Assume that the metric $g=g^E +h$ on $B_\rho$ is given. Then there
  exist $\eps_0>0$, $r_0<\rho$ and a constant $C<\infty$ such that for all surfaces $\Sigma
  \subset B_r$ as in the statement of
  theorem~\ref{thm:integral-estimate1} with $\eps<\eps_0$ and $r<r_0$ and for all smooth
  forms $\phi$ on $\Sigma$ we have
  \begin{equation*}
    \|\phi\|_{L^\infty(\Sigma)}^4
    \leq
    C \|\phi\|^2_{L^2(\Sigma)} \int_\Sigma |\nabla^2\phi|^2 + |H|^4 |\phi|^2 \dmu.
  \end{equation*}
\end{lemma}
\begin{proof}
  This lemma is a variant of \cite[Lemma
  2.8]{Kuwert-Schatzle:2001}. The proof from there can be carried over
  to our situation, since it mainly relies on the
  Michael-Simon-Sobolev inequality which is also available in this
  situation, cf.\ lemma~\ref{thm:michael-simon-soboloev}.
\end{proof}
\begin{proposition}
  \label{thm:Hinv_bdd}
  Assume that the metric $g=g^E +h$ on $B_\rho$ is given. Then there
  exist $\eps_0>0$, $r_0<\rho$ and a constant $C<\infty$ such that for all surfaces $\Sigma
  \subset B_r$ as in the statement of
  theorem~\ref{thm:integral-estimate1} with $\eps<\eps_0$ and $r<r_0$ we have
  \begin{equation*}
    \|H^{-1}\|_{L^\infty(\Sigma)} \leq C|\Sigma|^{1/2}
  \end{equation*}
\end{proposition}
\begin{proof}
  % Note that proposition~\ref{thm:H_sup} only yields that
  % $\|H^{-1}\|_{L^\infty} \leq Cr$ so that a seperate estimate is
  % needed.
  %
  The idea is to apply lemma~\ref{thm:sup-norm-from-w2-2} to the
  function $H^{-1}$. We thus estimate
  \begin{equation}
    \label{eq:15}    
    \|H^{-1}\|_{L^2(\Sigma)}^2 \leq |\Sigma| \|H^{-1}\|_{L^\infty(\Sigma)}^2
  \end{equation}
  and calculate
  \begin{equation*}
    \nabla^2 (H^{-1}) = - H^{-2}\nabla^2 H + 2H^{-3}\nabla H \otimes
    \nabla H.
  \end{equation*}
  Thus
  \begin{equation}
    \label{eq:16}
    \begin{split}
      \int_\Sigma |\nabla^2 (H^{-1})|^2 \dmu
      &\leq
      C \int_\Sigma H^{-6} |\nabla H|^4 + H^{-4}|\nabla^2 H|^2 \dmu
      \\
      &\leq
      C \|H^{-1}\|_{L^\infty(\Sigma)}^2 \int_\Sigma \frac{|\nabla^2 H|^2}{H^2} +
      |\nabla \log H|^4 \dmu.
    \end{split}
  \end{equation}
  From \eqref{eq:7d} we get for $\eps$ and $r$ small enough
  \begin{equation*}
    \int_\Sigma |\nabla \log H|^4\dmu
      \leq
      C\left(\int_\Sigma |\nabla\log H|^2\dmu\right)
      \left(\int_\Sigma \frac{|\nabla^2H|^2}{H^2}+|\nabla A|^2\dmu\right)
   \end{equation*}
  By theorem~\ref{thm:main-estimate} and corollary \ref{thm:lambda-bounded} we therefore conclude
  \begin{equation}
\label{eq:16a}
    \int_\Sigma |\nabla\log H|^4 \dmu
    \leq
    C|\Sigma|^3.
  \end{equation}
  Together with equation~\eqref{eq:16} and theorem~\ref{thm:main-estimate} this yields
  \begin{equation}
    \label{eq:14}
    \int_\Sigma |\nabla^2 (H^{-1})|^2 \dmu
    \leq
    C \|H^{-1}\|_{L^\infty(\Sigma)}^2 |\Sigma|.
  \end{equation}
  Plugging estimates~\eqref{eq:14} and~\eqref{eq:15} into the estimate
  from lemma~\ref{thm:sup-norm-from-w2-2}, we find that
  \begin{equation*}
    \begin{split}
      \|H^{-1}\|_{L^\infty(\Sigma)}^4
      &\leq C|\Sigma|\|H^{-1}\|_{L^\infty(\Sigma)}^2 \big( 
      \|H^{-1}\|_{L^\infty(\Sigma)}^2 |\Sigma| + \int_\Sigma H^2 \dmu\big)
      \\
      &\leq
      C|\Sigma|^2 \|H^{-1}\|_{L^\infty(\Sigma)}^4 + C |\Sigma|
      \|H^{-1}\|_{L^\infty(\Sigma)}^2
      \\
      &\leq
      (C|\Sigma|^2+\tfrac12) \|H^{-1}\|_{L^\infty(\Sigma)}^4 + C|\Sigma|^2.
    \end{split}
  \end{equation*}
  If $r<r_0$ and thus $|\Sigma|$ is small, the first term on the right
  can be absorbed, and the claim follows.
\end{proof}
%%% Local Variables: 
%%% mode: latex
%%% TeX-master: "master"
%%% ispell-local-dictionary: "en_US"
%%% End: 

%% file: gradient.tex
\section{Proof of Theorem \ref{thm:intro1}}
\label{sec:gradient}
This section is devoted to the proof of theorem~\ref{thm:intro1}. Note
that the first claim already follows from
corollary~\ref{thm:lambda-bounded}, hence it remains to show the
second claim.

Throughout this section we assume that the surface $\Sigma$ in
question is of Willmore type with multiplier $\lambda$. We assume
further that $H>0$ on $\Sigma$, $\lambda \geq -\eps|\Sigma|^{-1}$ and
$\Sigma\subset B_r(0)$ where $\eps<\eps_0$ and $r<r_0$. Here $\eps_0>0$ and $r_0>0$ are
chosen so that all the estimates from
section~\ref{sec:a-priori-estimates} are applicable.

To get started we recall the splitting~\eqref{eq:17} of the Willmore
functional:
\begin{equation}
  \CW(\Sigma) = 8\pi(1-q(\Sigma)) + \CU(\Sigma) + \CV(\Sigma).
\end{equation}
Since the first term on the right is a topological constant we infer
that the variation of $\CW$, when $\Sigma$ is varied by the normal
vector field $f\nu$ for $f\in C^\infty(\Sigma)$, satisfies
\begin{equation*}
  \delta_f \CW(\Sigma)
  =
  \delta_f \CU(\Sigma) + \delta_f\CV(\Sigma).
\end{equation*}
Equation~\eqref{eq:4} implies that the variation of $\CW$ is given by
\begin{equation*}
  \delta_f\CW(\Sigma) = \lambda \int_\Sigma Hf\dmu,
\end{equation*}
whenever $\Sigma$ is of Willmore type with multiplier $\lambda$.
Thus on such a surface we have
\begin{equation}
  \label{eq:19}
  \lambda \int_\Sigma Hf\dmu
  =
  \delta_f \CU(\Sigma) + \delta_f\CV(\Sigma).
\end{equation}
We shall evaluate these terms when the normal velocity $f$ of the
variation is given by
\begin{equation}
  \label{eq:1}
  f = H^{-1} g(b,\nu),
\end{equation}
where $b\in\IR^3$ is a fixed vector with $|b|=1$. We start with the
left hand side of equation~\eqref{eq:19}.
\subsection{The left hand side of \eqref{eq:19}}
\label{sec:left-hand-side}
We have
\begin{equation}
  \label{eq:21}
  \int_\Sigma H f \dmu = \int_\Sigma g(b,\nu) \dmu.
\end{equation}
To evaluate this expression note that since $\Sigma$ is assumed to be
a topological sphere in $B_\rho$ it must bound a region $\Omega$. We
wish to estimate the volume of $\Omega$. To this end, we approximate
$\Sigma$ by a Euclidean sphere $S = S_{R^E}(a^E)$. With $x$ the
position vector field in $B_\rho$ we define the vector field
\begin{equation}
  \label{eq:37}
  X = x - a^E
\end{equation}
such that 
\begin{equation*}
  \div_{g^E} X = 3,
\end{equation*}
in $\Omega$. On $\Sigma$ we have
\begin{equation*}
  \begin{split}
    |X|
    &=
    |\psi - a^E|
    \leq
    \|\psi-\id_S \|_{L^\infty(S)} + \|\id_S - a^E\|_{L^\infty(S)}
    \\
    &\leq
    C |\Sigma|^{1/2}\big(1 + \|\Acirc^E\|_{L^2(\Sigma,g^E)}\big),
  \end{split}
\end{equation*}
since $\|\id_S - a^E\|_{L^\infty(S)} = R^E$. Here $\id_S:S\to B_\rho$
denotes the standard embedding of $S$ into $B_\rho$.

We integrate the relation~\eqref{eq:37} over $\Omega$ and use partial
integration to conclude that
\begin{equation*}
  3\Vol^E(\Omega) 
  =
  \int_\Sigma g^E(X, \nu^E) \dmu^E .
\end{equation*}
Replacing the integral over $\Sigma$ by an integral over $S$
introduces an error of the form
\begin{equation*}  
  \begin{split}
    &\left|\int_\Sigma  g^E(X,\nu^E) \dmu^E 
      -
      \int_S g^E(X,N) \dmu^E \right|
    \\
    &\quad
    \leq
    C\big(
    |\Sigma| \|\psi-\id\|_{L^\infty(S)}
    + |\Sigma|^{3/2} \|h^2 -1\|_{L^\infty(S)}
    + |\Sigma| \|\nu^E\circ\psi - N\|_{L^2(S)}\big)
    \\
    &\quad \leq
    C|\Sigma|^{3/2}\|\Acirc^E\|_{L^2(\Sigma, g^E)}.
  \end{split}
\end{equation*}
In view of lemma~\ref{thm:euclidean-umbilic} we thus obtain the estimate 
\begin{equation*}
  \left| \Vol^E(\Omega) - \frac{|\Sigma|^{3/2}}{6\pi^{1/2}} \right|
  \leq
  C|\Sigma|^{3/2}\big(\|\Acirc\|_{L^2(\Sigma)}+r^2\big).
\end{equation*}
The assumption~\eqref{eq:3} implies that for the volume elements of $g$
and $g^E$ we have that
\begin{equation*}
\big|\dvol_g-\dvol_{g^E}\big|\le C|x|^2.
\end{equation*}
Combining the last two estimates we get
\begin{equation}
  \label{eq:30a}
  \left| \Vol(\Omega) - \frac{|\Sigma|^{3/2}}{6\pi^{1/2}} \right|
  \leq
  C|\Sigma|^{3/2}\big(\|\Acirc\|_{L^2(\Sigma)}+r^2\big).
\end{equation}
Using corollary~\ref{thm:lambda-bounded} we finally conclude
\begin{equation}
  \label{eq:30}
  \left| \Vol(\Omega) - \frac{|\Sigma|^{3/2}}{6\pi^{1/2}} \right|
  \leq
  Cr^2|\Sigma|^{3/2}.
\end{equation}
The right hand side of~\eqref{eq:21} can be expressed as a
volume integral
\begin{equation*}
  \int_\Sigma g(b,\nu) \dmu
  =
  \int_\Omega \div_M b\, \dvol
\end{equation*}
and since $|\nabla b| \leq Cr$ we estimate
\begin{equation*}
  \left| \int_\Omega \div_M b \dvol \right|
  \leq
  C r \Vol(\Omega)
  \leq
  Cr|\Sigma|^{3/2}
\end{equation*}
Thus, since $\lambda$ is bounded in view of
corollary~\ref{thm:lambda-bounded}, we obtain
\begin{equation}
  \label{eq:20}
  \left|\lambda \int_\Sigma H f \dmu\right|
  \leq
  Cr|\Sigma|^{3/2}.
\end{equation}
\subsection{The variation of $\CU$}
A fairly straight forward calculation shows that the variation of
$\CU$ is given by
\begin{equation}
  \label{eq:18}
  \delta_f \CU(\Sigma)
  =-
  \int_\Sigma 2 \la \Acirc, \nabla^2 f\ra + 2 f \la \Acirc, \Ric^T\ra
  + fH |\Acirc|^2 \dmu,
\end{equation}
where $\Ric^T$ denotes the tangential projection of the Ricci
curvature of $M$ onto $\Sigma$. With $f$ as in equation~\eqref{eq:1},
the second and third term are easily bounded as follows
\begin{equation}
  \label{eq:24}
  \begin{split}
    &\int_\Sigma 2 f \la \Acirc, \Ric^T\ra + fH |\Acirc|^2 \dmu
    \\
    &\quad
    \leq
    C |\Sigma|^{1/2} \sup_\Sigma |f| \left(\int_\Sigma
      |\Acirc|^2\dmu\right)^{1/2}
    + C\int_\Sigma
    |\Acirc|^2\dmu
    \leq
    C|\Sigma|^2
  \end{split}
\end{equation}
where we used the fact that $|g(b,\nu)|\leq C$ together with
corollary~\ref{thm:lambda-bounded} and proposition~\ref{thm:Hinv_bdd}.

To treat the first term in~\eqref{eq:18}, we calculate the first
and second derivatives of $f$. Choosing a local ON-frame
$\{e_1,e_2\}$ on $\Sigma$ we obtain
\begin{equation}
  \label{eq:22}
  \nabla_{e_i} f
  =
  H^{-1} g(\nabla_{e_i} b, \nu) + H^{-1} g(b, e_j) A_i^j - H^{-2}
  \nabla_{e_i} H g(b,\nu)
\end{equation}
and thus in view of the estimates from theorem~\ref{thm:main-estimate}, proposition~\ref{thm:Hinv_bdd}
and the fact that $|\nabla b| \leq Cr$, we find
\begin{equation*}
  \int_\Sigma |\nabla f|^2 \dmu
  \leq
  Cr^2.
\end{equation*}
Differentiating equation~\eqref{eq:22} once more, we obtain
\begin{equation}
  \label{eq:23}
  \begin{split}
    &\nabla_{e_i}\nabla_{e_j} f
    \\
    &=
   - A_i^k A_{jk} f
    +2 H^{-3}\nabla_{e_i} H \nabla_{e_j} H g(b,\nu)
    - H^{-2}\nabla^2_{i,j} H g(b,\nu)
    \\
    &\quad
    + H^{-1}\big(g(\nabla_{e_i}\nabla_{e_j} b , \nu)
    + g(\nabla_{e_i} b, e_k) A_j^k
    + g(\nabla_{e_j} b, e_k) A_i^k
    + \nabla_{e_j} A_i^k g(b, e_k)
    \big)
    \\
    &\quad
    - H^{-2}\big(\nabla_{e_i} H(g(\nabla_{e_j} b,\nu)+g(b,e_k)A^k_j)
    +\nabla_{e_j} H (g(\nabla_{e_i}b,\nu)+g(b,e_k)A^k_i)\big).
  \end{split}
\end{equation}
Our goal is to estimate $\int_\Sigma \la \Acirc,\nabla^2
f\ra\dmu$ so that all we need of $\nabla^2 f$ is its traceless
part. Note that the largest term in expression~\eqref{eq:23} is the
first one on the right hand side. Its contribution consists mainly of
the trace part. When removing the trace, we find that we can estimate
\begin{equation*}
  \begin{split}
    |(\nabla^2 f)^\circ|
    &\leq
    C\big(
    |A||\Acirc||f|
    + H^{-1} |\nabla^2 b|
    + H^{-1} |A| |\nabla b|
    + H^{-1} |\nabla A|
    \\
    &\qquad\ 
    + H^{-2} |\nabla H| |\nabla b|
    + H^{-2} |\nabla H| |A|
    + H^{-2} |\nabla^2 H|
    + H^{-3} |\nabla H|^2\big).
  \end{split}
\end{equation*}
In view of the fact that $|\nabla b| \leq Cr$ and $|\nabla^2 b| \leq
C$, and using the estimates from theorem~\ref{thm:main-estimate}, corollary~\ref{thm:lambda-bounded} and
proposition~\ref{thm:Hinv_bdd}, we infer that
\begin{equation}
  \label{eq:34}
  \int_\Sigma |(\nabla^2 f)^\circ|^2 \dmu
  \leq
  C r^2|\Sigma|,
\end{equation}
Here we also used that by the estimates from section $3$, \eqref{eq:16a} and the fact that $|A|^2=\frac12 H^2+|\Acirc|^2$ we have
\begin{equation*}
  \begin{split}
    \int_\Sigma H^{-4}|\nabla H|^2|A|^2\dmu
    &\le
    C\left(\int_\Sigma |\nabla \log H|^4\dmu\right)^{1/2}
    \!\!\!\left(\int_\Sigma H^{-4}(H^4+|\Acirc|^4)\dmu\right)^{1/2}
    \\
    &\le
    C|\Sigma|^{3/2}\big(|\Sigma|^{1/2}+|\Sigma|^{3/2}\big)
    \le  C |\Sigma|^2.
  \end{split}
\end{equation*}
In view of the Cauchy-Schwarz inequality,
corollary~\ref{thm:lambda-bounded} and estimate~\eqref{eq:34} we infer that
\begin{equation}
  \label{eq:25}
  \int_\Sigma \la \Acirc,\nabla^2 f\ra\dmu
  \leq
  \left(\int_\Sigma |\Acirc|^2 \dmu\right)^{1/2}
  \left(\int_\Sigma |(\nabla^2 f)^\circ|^2 \dmu\right)^{1/2}
  \leq Cr|\Sigma|^{3/2}.
\end{equation}
Collecting the estimates~\eqref{eq:24} and~\eqref{eq:25} implies the
desired bound on~\eqref{eq:18}, namely
\begin{equation}
  \label{eq:26}
  |\delta_f\CU(\Sigma)|
  \leq
  Cr|\Sigma|^{3/2}.
\end{equation}
\subsection{The variation of $\CV(\Sigma)$}
In section 5.5 of \cite{Lamm-Metzger-Schulze:2009} the
following expression for $\delta_f\CV(\Sigma)$ was derived:
\begin{equation}
  \label{eq:27}
  \delta_f\CV(\Sigma)
  =
  \int_\Sigma - f H G(\nu,\nu) - \tfrac{1}{2} f H \Scal +
  2 f \la\Acirc,G^T\ra - 2\omega(\nabla f) \dmu
\end{equation}
where as before $G = \Ric - \half g$ denotes the Einstein tensor of
$M$ and $\omega = \Ric(\nu,\cdot)^T$. Recall that we chose $f = H^{-1}
g(b,\nu)$ above. In the expression~\eqref{eq:22} we split $A= \Acirc +
\half H\gamma$ and obtain
\begin{equation*}
  \nabla_{e_i} f
  =
  \half g(b,e_i)
  + H^{-1}g(\nabla_{e_i} b,\nu)
  + H^{-1}\Acirc_i^j g(b,e_j)
  - H^{-2}\nabla H g(b,\nu).
\end{equation*}
Plugging this expression into equation~\eqref{eq:27} yields 
\begin{equation}
  \label{eq:28}
  \begin{split}
    \delta_f\CV(\Sigma)
    &=
    \int_\Sigma
    - G(b,\nu)
    - \tfrac{1}{2} g(b,\nu) \Scal
    +  2 f \la\Acirc,G^T\ra
    \\
    &\qquad
    - 2 w(e_i)\big(H^{-1}g(\nabla_{e_i} b,\nu)
    + H^{-1}\Acirc_i^j g(b,e_j)
    - H^{-2}\nabla H g(b,\nu)\big) \dmu.  
  \end{split}
\end{equation}
Using theorem~\ref{thm:main-estimate}, corollary~\ref{thm:lambda-bounded} and
proposition~\ref{thm:Hinv_bdd} we estimate
\begin{equation}
  \label{eq:29}
  \begin{split}
    &\Bigg|\int_\Sigma
      2 f \la\Acirc,G^T\ra
      - 2 w(e_i)\big(H^{-1}g(\nabla_{e_i} b,\nu)
      \\
      &\qquad
      + H^{-1}\Acirc_i^j g(b,e_j)
      - H^{-2}\nabla H g(b,\nu)\big) \dmu
    \Bigg|
    \leq
    Cr|\Sigma|^{3/2}.
  \end{split}
\end{equation}
In \cite{Lamm-Metzger-Schulze:2009} the Pohozaev-Identity was used to
estimate the term $ \int_\Sigma G(b,\nu)\dmu$. This is not really
necessary here, the following simpler approach is sufficient. Recall
that the divergence of $G$ with respect to the $g$-metric is
zero due to the Bianchi-identity. Define the vector field $X$ by the requirement that $g(X,Y)=
G(b,Y)$ for all vector fields $Y$ on $B_\rho$. Then the fact that $G$
is divergence free implies that
\begin{equation}
  \label{eq:31}
  \div_M X = \la G, \nabla b\ra.
\end{equation}
In section~\ref{sec:left-hand-side} we used that $\Sigma$ bounds a
region $\Omega$ with $\Vol(\Omega) \leq C|\Sigma|^{3/2}$. To proceed,
we integrate the relation~\eqref{eq:31} over $\Omega$ and after integration by
parts we get (recall that $|\nabla b|\le Cr$)
\begin{equation*}
  \left|\int_\Sigma G(b,\nu)\dmu\right|
  =
  \left|\int_\Omega \div_M X \dvol\right|
  \leq
  Cr \Vol(\Omega)
  \leq
  Cr|\Sigma|^{3/2}.
\end{equation*}
In combination with equation~\eqref{eq:28} and estimate~\eqref{eq:29}
we infer that
\begin{equation*}
  \left| \delta_f\CV(\Sigma) + \half \int_\Sigma g(b,\nu) \Scal
    \dmu\right|
  \leq
  Cr|\Sigma|^{3/2}.
\end{equation*}
The final task is to estimate $\int_\Sigma g(b,\nu) \Scal \dmu$. As
before, we express this surface integral as a volume integral. To this
end, we consider the vector field $X = \Scal\, b$ and calculate
\begin{equation*}
  \div_M X = g(b,\nabla\Scal) + \Scal \div_M b.
\end{equation*}
Since $\nabla\Scal = \nabla\Scal(0) + O(r)$, $\nabla b = O(r)$ and
$g= g^E + O(r^2)$ we infer
\begin{equation*}
  \div_M X = g^E(b,\nabla \Scal(0)) + O(r).
\end{equation*}
Thus we can calculate
\begin{equation*}
  \int_\Sigma g(b,\nu)\Scal \dmu
  =
  \int_\Omega \div_M X \dvol
  =
  \Vol(\Omega) g^E(b,\nabla \Scal(0)) + O(r\Vol(\Omega)).
\end{equation*}
This finally leaves us with the estimate
\begin{equation}
  \label{eq:32}
  \left| \delta_f\CV(\Sigma)
    + \half \Vol(\Omega) g^E(b,\nabla \Scal(0))
  \right|
  \leq
  Cr|\Sigma|^{3/2}.
\end{equation}
\subsection{The conclusion}
To prove theorem~\ref{thm:intro1}, we combine the results from the
previous sections. Combining equation~\eqref{eq:19} with the
estimates~\eqref{eq:20} and~\eqref{eq:26} yields that for $f$ as
in~\eqref{eq:1} we have
\begin{equation*}
  \left|\delta_f \CV(\Sigma)\right|
  \leq
  Cr|\Sigma|^{3/2}.
\end{equation*}
In combination with~\eqref{eq:32} this gives
\begin{equation*}
  \left| \Vol(\Omega) g^E(b,\nabla \Scal(0)) \right|
  \leq
  Cr|\Sigma|^{3/2},
\end{equation*}
and since $\Vol(\Omega) \geq C^{-1}|\Sigma|^{3/2}$ we infer
\begin{equation*}
  \left| g^E(b,\nabla \Scal(0)) \right|
  \leq
  Cr.
\end{equation*}
Setting $b = \nabla \Scal(0) / |\nabla \Scal(0)|^E$ finally shows
that
\begin{equation*}
  |\nabla \Scal(0)|^E
  \leq
  Cr.
\end{equation*}
Since we can let $r\to 0$ by the assumptions of
theorem~\ref{thm:intro1}, we infer the claim, namely that $\nabla
\Scal(0)=0$.

%%% Local Variables: 
%%% mode: latex
%%% TeX-master: "master"
%%% ispell-local-dictionary: "en_US"
%%% End: 

%% file: expansion.tex
\section{Expansion of the Willmore functional}
\label{sec:expans-willm-funct}
In this section we calculate the expansion of the Willmore functional
on small surfaces using the estimates from
section~\ref{sec:a-priori-estimates}.  We wish to emphasize here that
similar expansions for the Willmore functional have been computed
previously for geodesic spheres in \cite[Section 3]{fan-shi-tam:2009},
where also the subsequent term in the expansion is calculated, and for
perturbations of geodesic spheres in \cite{Mondino:2008}.

The calculation here has the advantage that it works under much more
general conditions. Namely we have the following theorem which holds
in particular for surfaces as in theorem~\ref{thm:intro1} due to the
estimates of lemma~\ref{thm:initial-estimate} and
corollary~\ref{thm:lambda-bounded}.
\begin{theorem}
  \label{thm:expansion}
  Let $g=g^E+h$ on $B_\rho$ be given and let $c<\infty$ be a
  constant. Then there exists a constant $C$ depending only on $c$,
  $\rho$ and $h_0$ as in equation~\eqref{eq:3}, such that the
  following holds.

  Let $\Sigma\subset B_r$ be a spherical surface with $r<\rho$ such
  that
  \begin{equation*}
    \CU(\Sigma) \leq cr|\Sigma|
    \qquad\text{and}\qquad
    |\Sigma|\le cr.
  \end{equation*}
  Then the following estimate holds:
  \begin{equation*}
    \left| \CW(\Sigma) - 8\pi + \frac{|\Sigma|}{3}\Scal(0) \right|
    \leq Cr|\Sigma|.
  \end{equation*}  
\end{theorem}
\begin{proof}
  We use the Gauss equation to express the Willmore functional as in
  equation~\eqref{eq:17}
  \begin{equation*}
    \CW(\Sigma) = 8\pi + \CU(\Sigma) + \CV(\Sigma).
  \end{equation*}
  By the first assumption the term $\CU(\Sigma)$ is a lower order
  term and can be neglected. Furthermore
  \begin{equation*}
    \CV(\Sigma)
    =
    2\int_\Sigma G(\nu,\nu) \dmu
    =
    \int_\Sigma 2\Ric(\nu,\nu) - \Scal \dmu.
  \end{equation*}
  In view of the assumptions of the theorem, lemma~\ref{thm:calc-ric}
  implies that
  \begin{equation*}
    \left|\int_\Sigma \Ric(\nu,\nu)\dmu
    -
    \frac{|\Sigma|}{3} \Scal(0)
    \right|
    \leq
    Cr|\Sigma| .
  \end{equation*}
  Since $\Scal = \Scal(0) + O(r)$ we furthermore have 
  \begin{equation*}
    \int_\Sigma \Scal \dmu
    =
    |\Sigma| \Scal(0) + O(r|\Sigma|)
  \end{equation*}
  so that in combination
  \begin{equation*}
    \CV(\Sigma)
    =
    -\frac{|\Sigma|}{3}\Scal(0) + O(r|\Sigma|).
  \end{equation*}
  Altogether this yields 
  \begin{equation*}
    \CW(\Sigma) = 8\pi - \frac{|\Sigma|}{3}\Scal(0) + O(r|\Sigma|),
  \end{equation*}
  which is the desired expansion.
\end{proof}
\begin{corollary}
  Let $g$ be as in theorem~\ref{thm:expansion} and assume that
  $\Sigma$ is a spherical surface in $B_r$ satisfying
  \begin{equation*}
    \CU(\Sigma) \leq c|\Sigma|^2
    \qquad\text{and}\qquad
    |\Sigma|\le cr.
  \end{equation*}
  If $\Omega$ denotes the region bounded by $\Sigma$, then the Hawking
  mass of $\Sigma$,
  \begin{equation*}
    m_H(\Sigma)
    =
    \frac{|\Sigma|^{1/2}}{(16\pi)^{3/2}} \big(16\pi -  2\CW(\Sigma)\big)
  \end{equation*}
  satisfies
  \begin{equation*}
    \frac{m_H(\Sigma)}{\Vol(\Omega)}
    =
    \frac{\Scal(0)}{16\pi} + O(r).
  \end{equation*}
\end{corollary}
\begin{proof}
  This is a simple consequence of the expansion in
  theorem~\ref{thm:expansion} which holds also under the stronger
  assumption of the corollary. In addition we use the fact that
  the volume of $\Omega$ satisfies
  \begin{equation*}
    \left|
      \Vol(\Omega) - \frac{|\Sigma|^{3/2}}{6\pi^{1/2}}
    \right|
    \leq
    Cr|\Sigma|^{3/2}
  \end{equation*}
  which follows from equation~\eqref{eq:30a}.
\end{proof}

%%% Local Variables: 
%%% mode: latex
%%% TeX-master: "master"
%%% End: 